\documentclass[leqno,11pt]{amsart}

\usepackage{amsfonts}
\usepackage{amsmath}
\usepackage{amsfonts}
\usepackage{amssymb}
\usepackage{amscd}
\newtheorem{theorem}{Theorem}
\newtheorem*{acknowledgement*}{Acknowledgement}

\newtheorem{corollary}[theorem]{Corollary}

\newtheorem*{example*}{Example}

\newtheorem{lemma}[theorem]{Lemma}

\newtheorem{proposition}[theorem]{Proposition}
\newtheorem{remark}[theorem]{Remark}

\begin{document}

\title{A remark on Einstein warped products}

\author{Michele Rimoldi}
\address{Dipartimento di Matematica\\
Universit\`a degli Studi di Milano\\
via Saldini 50\\
I-20133 Milano, ITALY}
\email{michele.rimoldi@unimi.it}



\begin{abstract}
We prove triviality results for Einstein warped products with non-compact bases. These extend previous work by D.-S. Kim and Y.-H. Kim. The proofs, from the viewpoint of ``quasi-Einstein manifolds'' introduced by J. Case, Y.-S. Shu and G. Wei, rely on maximum principles at infinity and Liouville-type theorems.
\end{abstract}

\maketitle

\section{Introduction}

The main purpose of this note is to prove the following triviality result  for Einstein warped products which extends, to the case
of non-compact bases, a recent theorem by D.-S. Kim and Y.-H. Kim, \cite{KK}.

\begin{theorem}\label{main_th}
Let $N^{n+m}=M^{n}\times_{u}F^{m}$, $m>1$, be a complete Einstein warped product with non-positive scalar curvature $^{N}S\leq0$, warping function $u(x)=e^{-\frac{f\left(  x\right)  }{m}}$ satisfying $\inf_M f=f_*>-\infty$ and complete Einstein fibre $F$.
Then $N$ is simply a Riemannian product if either one of the following further conditions is satisfied:

\begin{enumerate}
\item[(a)] $f$ has a local minimum.

\item[(b)] the base manifold $M$ is complete and non-compact, the warping
function satisfies $\int_M |f|^p e^{-\frac{f}{m}}d\mathrm{vol}<+\infty$, for
some $1<p<+\infty$, and $f\left(  x_{0}\right)  \leq0$ for some point $x_{0}\in
M$.
\end{enumerate}
\end{theorem}

Note that, in case $M$ is compact, from the point (a) we recover the main result in \cite{KK}.

Our proof of Theorem \ref{main_th} will rely on the
link between Einstein warped product metrics and the so called ``quasi-Einstein metrics''
recently introduced by J. Case, Y.-S. Shu and G. Wei, \cite{CSW}.
In the  spirit of \cite{PRiS}, i.e. using methods from stochastic analysis and $L^p$-Liouville type theorems, we shall  prove scalar curvature estimates and triviality results
for a complete quasi-Einstein manifold that largely extend previous theorems
in \cite{CSW}. Whence, the main theorem will follow immediately.

In a final section, using similar techniques, we extend another triviality result for Einstein warped products obtained in the very recent \cite{C}. A non-existence result is also discussed.

\section{Quasi-Einstein manifolds}
Consider the weighted manifold $(M^n,g_M, e^{-f}d\rm{vol})$, where $M$ is a complete $n$-dimensional Riemannian manifold, $f$ is a smooth real valued function on $M$ and $d\rm{vol}$ is the Riemannian volume density on $M$. A natural extension of the Ricci tensor to weighted manifolds is the $m$-Bakry-Emery Ricci tensor
\[
\ Ric_f^m=Ric+Hessf-\frac{1}{m}df\otimes df,\quad\textrm{for}\quad0<m\leq\infty.
\]
When $f$ is constant, this is the usual Ricci tensor and when $m=\infty$ this is the Ricci Bakry-Emery tensor $Ric_f$.
We call a metric $m$-quasi-Einstein if the $m$-Bakry-Emery Ricci tensor satisfies the equation
\begin{equation}\label{QE}
Ric_f^m=\lambda g_M,
\end{equation}
for some $\lambda\in\mathbb{R}$. This equation is especially interesting in that when $m=\infty$ it is exactly the gradient Ricci soliton equation.  When f is constant, it gives the Einstein equation and we call the quasi-Einstein metric trivial. When $m$ is a positive integer, it corresponds to warped product Einstein metrics.

Indeed, in \cite{CSW}, elaborating on \cite{KK}, it is observed the following characterization of quasi-Einstein metrics.
\begin{theorem}\label{CSW_warped}
Let $M^n\times_{u}F^m$ be an Einstein warped product with Einstein constant $\lambda$, warping function $u=e^{-\frac{f}{m}}$ and Einstein fibre $F^m$.
Then the weighted manifold $(M^n, g_M, e^{-f}d\rm{vol})$ satisfies the quasi-Einstein equation \eqref{QE}. Furthermore the Einstein constant $\mu$ of the fibre satisfies
\begin{equation}\label{EinsteinConstants}
\Delta f-\left|\nabla f\right|^2=m\lambda-m\mu e^{\frac{2}{m}f}.
\end{equation}
Conversely if the weighted manifold $(M^n,g_M,e^{-f}d\rm{vol})$ satisfies \eqref{QE}, then $f$ satisfies \eqref{EinsteinConstants} for some constant $\mu\in\mathbb{R}$. Consider the warped product $N^{n+m}=M^n\times_u F^m$, with $u=e^{-\frac{f}{m}}$ and Einstein fibre $F$ with $^F Ric=\mu g_F$. Then $N$ is Einstein with $^N Ric=\lambda g_N$.
\end{theorem}

\section{Scalar curvature estimates}

In this section, in the same spirit of Theorem 3 of \cite{PRiS}, we generalize the scalar curvature estimates in Proposition 3.6 of \cite{CSW} to quasi-Einstein manifolds with non-constant scalar curvature. Possible rigidity at the endpoints is also discussed.

\begin{theorem}\label{scal_est}
Let $(M^{n}, g_M, e^{-f}d\rm{vol}) $ be a geodesically complete $m$-quasi- Einstein manifold, $1< m<+\infty$, with scalar curvature $S$ and let $S_*= \inf_M S$.\\
\begin{enumerate}
	\item [(a)]If $\lambda>0$, then $M$ is compact and
	\begin{equation}\label{Scal_lambda+}
	\frac{n(n-1)}{m+n-1}\lambda< S_*\leq n\lambda.
	\end{equation}
	Moreover $S_*\neq n\lambda$ unless $M$ is Einstein. 
	\item[(b)] If $\lambda=0$ and $\inf_M f=f_*>-\infty$ then $S_*=0$. Moreover, either $S>0$ or $S(x)\equiv0$. In this latter case, either $f$ is constant (and $M$ is trivial) or $M$ is isometric to the Riemannian product $\mathbb{R}\times\Sigma$ where $\Sigma$ is a  Ricci-flat, totally geodesic hypersurface.
	\item[(c)]If $\lambda<0$ and $\inf_M f=f_*>-\infty$, then
	\begin{equation}\label{Scal_lambda-}
	n\lambda\leq S_*\leq\frac{n(n-1)}{m+n-1}\lambda
	\end{equation}
	and $S(x)>n\lambda$ unless $M$ is Einstein.
\end{enumerate}
\end{theorem}
The proof of Theorem \ref{scal_est} will require the following formula obtained in \cite{CSW}, which generalizes to the case $m<+\infty$ similar formulas for Ricci solitons ($m=+\infty$) obtained previously by P. Petersen and W. Wylie, \cite{PW_rig}. Following the terminology introduced in \cite{PW_class}, the $f$-Laplacian on the weighted manifold $(M,g_M,e^{-f}d\rm{vol})$ is the diffusion type operator defined by
$\Delta_fu=e^f \rm{div}(e^{-f}\nabla u).$ It is clearly a symmetric operator on $L^2(M,e^{-f}d\rm{vol})$.
\begin{lemma}\label{formulae}
Let $Ric_f^m=\lambda g_M$, for some $\lambda\in \mathbb{R}$ and  $m<+\infty$. Set $\tilde{f}=\frac{m+2}{m}f$. Then

\begin{equation}\label{eq_Scal1}
\frac{1}{2}\Delta_{\tilde{f}} S= -\frac{m-1}{m}|Ric-\frac{1}{n}Sg_M|^2-\frac{m+n-1}{mn}(S-n\lambda)(S-\frac{n(n-1)}{m+n-1}\lambda).
\end{equation}

\end{lemma}

\begin{proof}[Proof (of Theorem \ref{scal_est})]
First of all, we show that $\inf_M S>-\infty$. According to Qian version of Myers' theorem this is obvious if $\lambda>0$ because $M$ is compact, see also the Appendix. In the general case $\lambda\in\mathbb{R}$ we proceed as follows. Since
\[
\ -\left|Ric-\frac{1}{n}Sg_M\right|^2=-\left|Ric\right|^2+\frac{S^2}{n},
\]
from \eqref{eq_Scal1} we obtain
\begin{align}\label{ineq_scal}
\frac{1}{2}\Delta_{\tilde{f}}S=&-\frac{m-1}{m}|Ric|^2-\frac{1}{m}S^2+\frac{m+2n-2}{m}S\lambda-\frac{n(n-1)}{m}\lambda^2.\\
\leq&-\frac{1}{m}S^2+\frac{m+2n-2}{m}\lambda S.\nonumber
\end{align}
Let $S_{-}(x)=\max\{-S(x),0\}$. Then
\begin{equation}\label{ineq_main}
\Delta_{\tilde{f}}S_-\geq\frac{2}{m}S_-^2+\frac{2(m+2n-2)}{m}\lambda S_-.
\end{equation}
Observe now that from Qian's estimates of weighted volumes (\cite{Q}, see also section 2 in \cite{MRS} and references therein), since $vol_{\tilde{f}}(B_r)\leq e^{-\frac{2}{m}f_*}vol_f(B_r)$, we can apply the ``a-priori'' estimate in Theorem 12 of \cite{PRiS} to inequality \eqref{ineq_main} on the complete weighted manifold $(M,g_M,e^{-\tilde{f}}d\rm{vol})$ and we obtain that $S_{-}$ is bounded from above, or equivalently, $S_{*}=\inf_M S>-\infty$.
Again from the volume estimates in \cite{Q} and by Theorem 9 in \cite{PRiS} applied to $(M,g_M,e^{-\tilde{f}}d\rm{vol})$, the weak maximum principle at infinity for the $\tilde{f}$-laplacian holds on $M$. This produces a sequence $\{x_k\}$ such that $\Delta_{\tilde{f}}S(x_{k})\geq-\frac{1}{k}$ and $S(x_{k})\rightarrow S_{*}$. Taking the $\liminf$ in \eqref{eq_Scal1} along $\{x_{k}\}$ shows that, for $m>1$,
\begin{equation}\label{ineq_S*}
0\leq-\frac{m+n-1}{mn}( S_*-n\lambda)(S_*-\frac{n(n-1)}{m+n-1}\lambda).
\end{equation}
We now distinguish three cases.\smallskip

\noindent (a) Assume $\lambda>0$, so that $M$ is compact. Equation (\ref{ineq_S*}) yields $\frac{n(n-1)}{m+n-1}\lambda\leq S_{*}\leq n\lambda$. Assume now that $S_{*}=n\lambda>0$. Then $S\geq n\lambda\geq\frac{n(n-1)}{m+n-1}\lambda$ and from \eqref{eq_Scal1} we get
\[
\ \frac{1}{2}\Delta_{\tilde{f}}S\leq-\frac{m+n-1}{mn}(S-n\lambda)(S-\frac{n(n-1)}{m+n-1}\lambda)\leq0.
\]
Since $M$ is compact, $S$ must be constant. Hence $S=S_{*}=n\lambda$. Substituting in \eqref{eq_Scal1} we obtain that $Ric=\frac{1}{n}Sg_M$ and thus that $M$ is Einstein.

Now we show that $S_*>\frac{n(n-1)}{m+n-1}\lambda$. Indeed, suppose that $S$ attains its minimum $\frac{n(n-1)}{m+n-1}\lambda$. Since the non-negative function $v(x)=S(x)-\frac{n(n-1)}{m+n-1}\lambda$ satisfies
\begin{equation*}
\frac{1}{2}\Delta_{\tilde{f}}v\leq-\frac{m+n-1}{mn}\,v^2+\lambda v\leq+\lambda v,
\end{equation*}
and $v$ attains its minimum $v(x_0)=0$, it follows from the minimum principle, (see p. 35 in \cite{GT}), that $v$ vanishes identically. Hence $S\equiv \frac{n(n-1)}{m+n-1}\lambda$ is constant and, substituting in \eqref{eq_Scal1}, we get that $M$ is Einstein with
\[
\ Ric=\frac{n-1}{m+n-1}\lambda g_M.
\]
Using this information into \eqref{QE} we obtain that
\begin{equation*}
\ Hess(f)=\frac{1}{m}|\nabla f|^2+\frac{m}{m+n-1}\lambda g_M>0. 
\end{equation*}
But this is clearly impossible because $M$ is compact.
\smallskip

\noindent (b) Assume $\lambda=0$. From  (\ref{ineq_S*}) we conclude that $S_*=0$.
Note that, according to (\ref{eq_Scal1}), $\Delta_{\tilde{f}}S\leq0$. Therefore, by the minimum principle, either $S(x)>0$ on $M$ or $S(x)\equiv0$. In this latter case,  substituting in (\ref{eq_Scal1}), we obtain that $M$ is Ricci flat and the $m$-quasi Einstein equation reads $Hess(f)-\frac{1}{m}df \otimes df = 0.$ Therefore, either $f$ is constant and $M$ is Einstein, or the non constant function $u=e^{-\frac{f}{m}}$ satisfies $\rm{Hess}(u)=0$. A Cheeger-Gromoll type argument now shows that $M$ is isometric to the Riemannian product $\mathbb{R}\times \Sigma$ along the  Ricci flat, totally geodesic hypersurface $\Sigma$ of $M$.

\noindent (c) Assume $\lambda<0$. From (\ref{ineq_S*}) we deduce that $n\lambda\leq S_{*}\leq\frac{n(n-1)}{m+n-1}\lambda$. Suppose that $S(x_0)=n\lambda<0$ for some $x_0\in M$. Since the non-negative function $w(x)=S(x)-n\lambda$ satisfies
\begin{equation*}
\frac{1}{2}\Delta_{\tilde{f}}w\leq-\frac{m+n-1}{mn}\,w^2-\lambda w\leq-\lambda w,
\end{equation*}
and $w$ attains its minimum $w(x_0)=0$, it follows from the minimum principle that $w$ vanishes identically. Hence $S\equiv n\lambda$ is constant and substituting in \eqref{eq_Scal1} we get that $M$ is Einstein.
\end{proof}

\section{Triviality results under $L^{p}$ conditions}
It is well known that steady or expanding compact Ricci solitons are necessarily trivial. The same result is proven in \cite{KK} for quasi-Einstein metrics on compact manifolds with finite $m$. For Ricci solitons a generalization to the complete non-compact setting is obtained in \cite{PRiS}.

In this section using the scalar curvature estimates of Theorem \ref{scal_est}, we get triviality for (not necessarily compact) quasi-Einstein metrics with $m<+\infty$,  $\lambda\leq0$.
\begin{theorem}\label{Triv}
Let $(M^n,g_M,e^{-f}d\rm{vol})$ be a geodesically complete non-compact $m$-quasi-Einstein manifold,  $1\leq m<+\infty$. If the quasi-Einstein constant $\lambda$ is non-positive and $f$ satisfies, for some $1<p<+\infty$,
\begin{equation}\label{p_int}
f\in L^p(M, e^{-\frac{f}{m}}d\rm{vol}),
\end{equation}
and $\inf_M f=f_*>-\infty$, then either $f\equiv const\leq0$ and $M$ is Einstein or $f>0$. 
\end{theorem}

\begin{proof}(of Theorem \ref{Triv})
Tracing \eqref{QE} and letting $\hat{f}=\frac{1}{m}f$ we have that
\begin{equation}\label{triv1}
\Delta_{\hat{f}}f=n\lambda-S.
\end{equation}
Since $\lambda\leq0$ and $f_*>-\infty$, from \eqref{Scal_lambda-} of Theorem \ref{scal_est} we obtain that $\Delta_{\hat{f}}f\leq 0$. Applying Theorem 14 in \cite{PRiS} to $f_{-}=\max\{-f,0\}\in L^p(M,e^{-\hat{f}}d\rm{vol})$,  gives that $f_{-}$ is constant. Hence, if there exists a point $x_0\in M$ such that $f(x_0)\leq0$ then $f\equiv f(x_0)\leq0$.
\end{proof}
\begin{remark}\label{Dezhong}
\rm{From the proof it follows that if either $M$ is compact or $f$ attains its absolute minimum then $f\equiv const$. Actually, it was pointed out to us by Dezhong Chen that the same conclusion holds if we merely assume that $f$ attains a local minimum at some point $x_0\in M$. Indeed the following proposition holds.}
\end{remark}

\begin{proposition}\label{rmk_min}
Let $(M, g_M, e^{-f} d\rm{vol})$ be a geodesically complete non-compact $m$-quasi-Einstein manifold, $1< m<+\infty$. If the quasi-Einstein constant $\lambda$ is non positive and $f$ satisfies $f_*>-\infty$, then any local minimum of $f$ is actually an absolute minimum.
\end{proposition}
\begin{proof}
Assume that $f$ attains a local minimum $x_0\in M$. Evaluating \eqref{triv1} at $x_0$, we get
\[
\ S(x_0)\leq n\lambda.
\]
Hence, since $\lambda\leq 0$, by Theorem \ref{scal_est}, $M$ is Einstein and $S$ is identically $n\lambda$. Thus the quasi-Einstein equation \eqref{QE} reads
\begin{equation}\label{split}
\ Hess(f)=\frac{1}{m} df \otimes df.
\end{equation}
In particular $Hess(f)$ is positive semi-definite on $M$ and this implies the thesis.
\end{proof}

\section{Proof of the main theorem}

Putting together the results of the previous sections we easily obtain a proof of Theorem \ref{main_th}.

Indeed, according to Theorem \ref{CSW_warped}, $M$ is quasi-Einstein. Statement (a) follows immediately from Remark \ref{Dezhong} and Proposition \ref{rmk_min}. In case (b), since $(n+m)\lambda=\text{ }^{N}S\leq 0$, we get by Theorem \ref{Triv} that $f$, and so $u$, is a constant function.

\section{Other triviality results}
Another triviality result for Einstein warped products has been obtained by J. Case in \cite{C}. 

\begin{theorem}\label{Case}(Case)
Let $N^{n+m}=M^{n}\times_{u}F^{m}$ be a complete warped product with warping function $u(x)=e^{-\frac{f\left(  x\right)  }{m}}$, scalar curvature $^NS\geq 0$ and complete Einstein fibre $F$.
Then $N$ is simply a Riemannian product provided the base manifold $M$ is complete and the scalar curvature of $F$ satisfies $^{F}S\leq 0$.
\end{theorem}

In the following theorem we obtain the same conclusion in case the fibers have non-negative scalar curvature, up to assume an integrability condition on the warping function $u$. We observe that non-trivial examples with $^N S\leq 0$ and $^F S\geq 0$ are constructed in (\cite{B}, 9.118). Thus the integrabilty assumption is necessary.
\begin{theorem}\label{RicFlatLp}
Let $N^{n+m}=M^{n}\times_{u}F^{m}$ be a complete Einstein warped product with warping function $u(x)=e^{-\frac{f\left(  x\right)  }{m}}$, scalar curvature $^NS\leq 0$, and complete Einstein fibre $F$.
Then $N$ is simply a Riemannian product provided the base manifold $M$ is complete, the warping function satisfies $\int_Me^{-(\frac{p+m}{m})f}d\rm{vol}<+\infty$ for some $1<p<+\infty$, and the scalar curvature of $F$ satisfies $^{F}S\geq 0$. In this case $M$ and $F$ are Ricci flat and $M$ is compact.
\end{theorem}

Combining Theorem \ref{Case} and Theorem \ref{RicFlatLp} immediately gives the following

\begin{corollary}
Let $N$ be a complete Ricci flat warped product with complete Einstein fibre $F$ and warping function $u(x)=e^{-\frac{f\left(  x\right)  }{m}}$ satisfying $u\in L^p(M,e^{-f}d\rm{vol})$, for some $1<p<+\infty$. Then $N$ is simply a Riemannian product.
\end{corollary} 
\begin{proof}(of Theorem \ref{RicFlatLp})
Just observe that computing the $f$-laplacian of $u$ and using \eqref{EinsteinConstants} one obtains the following equation
\begin{equation}\label{EinstConst_u}
\ \Delta_f u=\mu u^{-1}-\lambda u+ \frac{u}{m^2}|\nabla f|^2.
\end{equation}
Thus, in our assumptions, we obtain that $\Delta_fu\geq0$. Since $0<u\in L^p(M,e^{-f}d\rm{vol})$, by Theorem 14 in \cite{PRiS}, we obtain the constancy of $u$. Up to a rescaling of the metric of $F$ we can suppose $u=1$.

Now, since the Riemannian product $M\times F$ is Einstein, both $M$ and $F$ are Einstein manifolds with the same Einstein constant. In particular, $^MS$ and $^FS$ have the same sign. By our assumption on the signs of $^NS$ and $^FS$ we thus obtain that both $M$ and $F$ are Ricci flat. Finally, since $u$ (and thus $f$) is constant, from the integrability condition we obtain that $vol(M)<+\infty$. Thus, by a result of Calabi-Yau, [Y], we obtain that $M$ must be compact.
\end{proof}

We end this section with a non-existence result. Recall that by the volume estimates in \cite{Q} and Theorem 9 in \cite{PRiS} the weak maximum principle for the $f$-laplacian holds on $(M, g_M, e^{-f}d\rm{vol})$ provided $Ric_f^m=\lambda g_M$ for some $\lambda\in\mathbb{R}$, $m<+\infty$.
\begin{theorem}\label{Non-Exist}
There is no complete Einstein warped product $N=M^n\times_uF^m$ with warping function $u=e^{-\frac{f}{m}}\in L^\infty(M)$, scalar curvature $^NS<0$ and Einstein fibre $F$ with $^FS\geq 0$.
\end{theorem}
\begin{proof}
Since $m\mu= ^FS\geq 0$, from \eqref{EinstConst_u}, we have that 
\begin{equation}\label{Linfty}
\Delta_fu\geq-u\lambda.
\end{equation}
Since, by assumption, $u$ satisfies $\sup_M u=u^*<+\infty$, by the weak maximum principle at infinity for the $f$-laplacian, there exists a sequence $\{x_k\}\subset M$ along which $u(x_k)\geq u^*-\frac{1}{k}$ and $\Delta_f u(x_k)\leq\frac{1}{k}$. Thus evaluating \eqref{Linfty} along $\{x_k\}$ and taking the limit as $k\rightarrow+\infty$ we obtain that $\lambda u^*\geq 0$ and since $u^*>0$ we cannot have $\lambda<0$.
\end{proof}

\section*{Appendix}
An extension of Myers' theorem to weighted manifolds with a positive lower bound on the $m$-Bakry-Emery Ricci tensor (m finite) is obtained by Qian in \cite{Q}. For generalizations of Myers' theorem in a different direction see \cite{M}.

In this section we extend Qian theorem by allowing some negativity of the $m$-Bakry-Emery Ricci tensor. The starting point of our considerations is the following Bochner formula for the $m$-Bakry-Emery Ricci tensor; see e.g. \cite{S}.

Let $u:M^n\rightarrow\mathbb{R}$ be a smooth function on a complete weighted manifold $(M^n,g_M,e^{-f}d\rm{vol})$ then
\begin{equation}\label{Bochner}
\frac{1}{2}\Delta_f\left|\nabla u\right|^2=\left|Hess(u)\right|^2+g_M( \nabla u, \nabla\Delta_fu)+ Ric_f^m(\nabla u, \nabla u)+\frac{1}{m}\left|g_M( \nabla f, \nabla u)\right|^2.
\end{equation}
Using this formula one obtains the following generalization of a well-known lemma which estimate the integral of Ricci along geodesics. The proof is modelled on \cite{Q}.
\begin{lemma}\label{main_lem}
Let $(M^n,g_M, e^{-f}d\rm{vol})$ be a complete weighted manifold, and consider the $m$-Bakry-Emery Ricci tensor $Ric_f^m$ for $m$ finite. Fix $o\in M$ and let $r\left(x\right)=dist\left(x,o\right)$. For any point $q \in M$, let $\gamma_{q}:\left[0,r\left(q\right)\right]\rightarrow M$ be a minimizing geodesic from $o$ to $q$ such that $\left|\dot{\gamma}_{q}\right|=1$. If $h\in Lip_{loc}\left(\mathbb{R}\right)$ is such that $h\left(0\right)=h\left(r\left(q\right)\right)=0$, then  for every $q\in M$, it holds
\begin{equation}\label{ineq_lemma}
0\leq \int^{r\left(q\right)}_{0}(m+n-1)\left(h^{\prime}\right)^{2}ds-\int^{r\left(q\right)}_{0}h^{2}Ric_f^m(\dot{\gamma}_q,\dot{\gamma}_q)ds.
\end{equation}
\end{lemma}
\begin{proof}
Fix a point $q\notin cut(o)$.
Straightforward computations show that
\begin{align}
\frac{(\Delta_f r)^2}{m+n-1}\leq&\frac{(\Delta r)^2}{n-1}+\frac{\left|g_M( \nabla f, \nabla r)\right|^2}{m},\label{A}\\
\left|Hess(r)\right|^2\geq&\frac{(\Delta r)^2}{n-1}\label{B}.
\end{align}
Using \eqref{A} and \eqref{B}, from the Bochner formula \eqref{Bochner} applied to the distance function $r(x)$ we obtain that
\[
\ 0\geq \frac{(\Delta_f r)^2}{m+n-1}+g_M( \nabla r, \nabla \Delta_f r)+ Ric_f^m(\nabla r, \nabla r).
\]
Evaluating this along a minimizing geodesic $\gamma_q$ such that $\left|\dot{\gamma}_q\right|=1$, we get
\begin{equation}\label{eq_1}
0\geq\frac{(\Delta_f r\circ\gamma_q)^2}{m+n-1}+\frac{d}{ds}(\Delta_f(r\circ\gamma_q))+Ric_f^m(\dot{\gamma}_q,\dot{\gamma}_q).
\end{equation}
If $h\in Lip_{loc}(\mathbb{R})$, $h\geq 0$, $h(0)=0$, multiplying \eqref{eq_1} by $h^2$ and integrating on $[0,t]$, we obtain
\begin{equation*}
0\geq\int_0^th^2\frac{(\Delta_f r\circ\gamma_q)^2}{m+n-1}ds+\int_0^t\frac{d}{ds}(\Delta_f r\circ\gamma_q)h^2+\int_0^th^2Ric_f^m(\dot{\gamma}_q,\dot{\gamma}_q).
\end{equation*}
Since $(\Delta_f r\circ\gamma_q)h^2\to0$ as $r\to0$, integrating by parts we have that
\begin{align*}
 0\geq&\int_0^t h^2\frac{(\Delta_f r\circ \gamma_q)^2}{m+n-1}ds+h^2(t)(\Delta_f r\circ\gamma_q)(t)\\&-2\int_0^t hh^{\prime}(\Delta_f r\circ\gamma_q)ds+ \int_0^th^2Ric_f^m(\dot{\gamma}_q,\dot{\gamma}_q)ds.
\end{align*}
Since
\[
\ -2hh^{\prime}(\Delta_f r\circ\gamma_q)\geq\frac{-h^2(\Delta_f r\circ\gamma_q)^2}{m+n-1}-(m+n-1)(h^{\prime})^2,
\]
we deduce that
\[
\ 0\geq h^2(t)(\Delta_f r\circ\gamma_q)-\int_0^t(m+n-1)(h^{\prime})^2ds+\int_0^tRic_f^m(\dot{\gamma}_q,\dot{\gamma}_q)h^2ds
\]
Thus, taking $t=r(q)$ and choosing $h$ such that $h^2(r(q))=0$, we get \eqref{ineq_lemma} for $q\notin cut(o)$. To treat the general case one can use the Calabi trick. Namely suppose that $q\in cut\left(o\right)$.
Translating the origin $o$ to $o_{\epsilon}=\gamma_{q}\left(\epsilon\right)$
so that $q\notin cut\left(o_{\epsilon}\right)$, using the triangle inequality
and, finally, taking the limit as $\epsilon\rightarrow 0$, one checks that
(\ref{ineq_lemma}) holds also in this case.
\end{proof}
From Lemma \ref{main_lem} some Myers' type results can be proven. Here we state the following which generalizes a theorem of G. J. Galloway, \cite{Gal-JDG}.

\begin{theorem}\label{Galloway-Myers}
Let $(M^n,g_M,e^{-f}d\rm{vol})$ be a complete weighted manifold. Given two different
points $p,q\in M$, let $\gamma_{p,q}$ be a minimizing geodesic
from $p$ to $q$ parameterized by arc length. Suppose that there
exist constants $c$ and $G\geq0$ such that for each pair of
points $p,q$ it holds
\[
Ric_f^m(\dot{\gamma}_{p,q},\dot{\gamma}_{p,q})|_{\gamma_{p,q}(t)}\geq (m+n-1)\left[c^2+\frac{d}{dt}\left(g\circ\gamma_{p,q}\right)\right],
\]
for some $C^1(M)$ function $g$ satisfing $\sup_M|g|\leq G$, $m<+\infty$. Then $M$ is compact and
\begin{equation}\label{estimate_diam}
\operatorname{diam}(M)\leq\frac{1}{c}\left[\frac{2G}{c}+\sqrt{\frac{4G^2}{c^2}+\pi^2}\right].
\end{equation}
\end{theorem}
\begin{proof}
Define $L$ to be the length of $\gamma_{p,q}$ between $p$ and $q$ and set $h(t):=\sin(\frac{\pi}{L}t)$. Compute
\[
\int_0^Lh^2(t)dt=\int_0^L\sin^2(\frac{\pi}{L}t)dt=\frac{L}{2};\quad\int_0^L\left.h'\right.^2(t)dt=\frac{\pi^2}{L^2}\int_0^L\cos^2(\frac{\pi}{L}t)dt=\frac{\pi^2}{2L}.
\]
Then, applying Lemma \ref{main_lem}, we have
\begin{align}\label{myers_ineq}
\frac{\pi^2(m+n-1)}{2L} =& \int_0^L(m+n-1)\left.h'\right.^2 \geq \int_0^Lh^2Ric_f^m(\dot{\gamma}_{p,q},\dot{\gamma}_{p,q})|_{\gamma_{p,q}}ds \\
\geq& c^2(m+n-1)\int_0^L h^2 + (m+n-1)\int_0^L h^2\frac{d}{dt}(g\circ\gamma_{p,q}) \nonumber\\
=& \frac{c^2(m+n-1)L}{2}+\left.(m+n-1)h^2g(\gamma_{p,q})\right|_0^L \nonumber\\
&- (m+n-1)\left[\int_0^{\frac{L}{2}}(\frac{d}{dt}h^2)(g\circ\gamma_{p,q})+\int_{\frac{L}{2}}^L(\frac{d}{dt}h^2)(g\circ\gamma_{p,q})\right] \nonumber\\
\geq& \frac{c^2(m+n-1)L}{2}-(m+n-1)G\left[\int_0^{\frac{L}{2}}(\frac{d}{dt}h^2)+\int_{\frac{L}{2}}^L\left|\frac{d}{dt}h^2\right|\right]\nonumber\\
\geq&\frac{c^2(m+n-1)L}{2}-2(m+n-1)G\nonumber
\end{align}
Finally, this latter can be written as
\[
c^2L^2-4GL-\pi^2\leq 0,
\]
which in turn implies \eqref{estimate_diam}, because $p$ and $q$ are arbitrary.
\end{proof}
Reasoning as in the classical case, (\cite{Gal-PAMS}, \cite{MRV}) the validity of \eqref{ineq_lemma} and an integration by parts shows that the compactness of $M$ depends on the behavior, and on the position of the zeros, of the solution of the differential equation along minimizing geodesics
\begin{equation}\label{eq_diff} -h^{\prime\prime}(t)-\frac{Ric_f^m(\dot{\gamma},\dot{\gamma})}{m+n-1}h(t)=0
\end{equation}
We are thus reduced to find sufficient condition on $Ric_f^m$ for which solutions of the differential equation \eqref{eq_diff} have a first zero at finite time. Minor changes to the proofs of the results contained in \cite{MRV} lead to similar compactness results in the weighted setting. In particular we state the following theorem in which a Myers' type conclusion is obtained assuming a nonpositive lower bound on $Ric_f^m$.
\begin{theorem}\label{th_main-B2}
Let $Ric_f^m\geq -(m+n-1)B^2$, for some constant $B\geq0$, $m<+\infty$. Suppose
there is a point $q\in M$ such that along each geodesic
$\gamma:[0,+\infty)\to M$ parameterized by arc length, with
$\gamma(0)=q$, it holds either
\begin{align}\label{ass_main-B2_0}
    \int_a^bt\frac{Ric_f^m(\dot{\gamma},\dot{\gamma})}{m+n-1}dt
    >
    B\left\{b+a\frac{e^{2Ba}+1}{e^{2Ba}-1}\right\}+\frac{1}{4}\log\left(\frac{b}{a}\right).
\end{align}
or
\begin{equation}\label{ass_main-B2}
\int_a^bt^{\alpha}\frac{Ric_f^m(\dot{\gamma},\dot{\gamma})}{m+n-1}(t)dt
>
B\left\{b^{\alpha}+a^{\alpha}\frac{e^{2Ba}+1}{e^{2Ba}-1}\right\}+\frac{\alpha^2}{4(1-\alpha)}\left\{a^{\alpha-1}-b^{\alpha-1}\right\}
\end{equation}
for some $0<a<b$ and $\alpha\neq1$. Then $M$ is compact.
\end{theorem}
\bigskip
\begin{acknowledgement*}
The author is deeply grateful to to his advisor Stefano Pigola for his guidance and constant encouragement during the
preparation of the manuscript.  The author would like to thank also Dezhong Chen for pointing out to him Proposition \ref{rmk_min}
\end{acknowledgement*}


\bigskip

\begin{thebibliography}{99}

\bibitem {B} A. L. Besse, \textsl{Einstein manifolds.} Springer-Verlag, Berlin-Heidelberg (1987).

\bibitem {C}J. Case, \textit{On the nonexistence of quasi-Einstein metrics.} To appear in Pacific J. Math. arXiv:0902.2226v3.                                                                                 
\bibitem {CSW}J. Case, Y.-S. Shu, G. Wei, \textit{Rigidity of Quasi-Einstein metrics.} arXiv:0805.3132v1.

\bibitem{Gal-JDG} G. J. Galloway, \textsl{A generalization of Myers' theorem and an application to relativistic cosmology.} J. Differential Geom. 14 (1979), no. 1, 105--116 (1980).

\bibitem{Gal-PAMS} G. J. Galloway, \textsl{Compactness criteria for Riemannian manifolds. } Proc. Amer. Math. Soc. 84 (1982), no. 1, 106--110.

\bibitem{GT} D.~Gilbarg, N.~Trudinger, \textit{Elliptic
Partial Differential Equations of Second Order}, second edition,
Springer-Verlag, Berlin (1983).

\bibitem{KK} D.-S. Kim, Y. H. Kim, \textsl{Compact Einstein warped product spaces with nonpositive scalar curvature. } Proc. Amer. Math. Soc. 131 (2003), no. 8, 2573--2576.

\bibitem {MRS}L. Mari, M. Rigoli, A. G. Setti, \textit{
Keller-Osserman conditions for diffusion-type operators on Riemannian
manifolds}. Jour. Funct. Anal. {\bf 258} (2010), 665-–712.

\bibitem {MRV} P. Mastrolia, M. Rimoldi and G. Veronelli,
\textit{Myers' type theorems and some related oscillation results}. arXiv:1002.2076.

\bibitem {M} F. Morgan, \textit{Myers' theorem with density.} Kodai Math. J.  \textbf{29} (2006), n. 3, 455--461.

\bibitem {PW_class}P. Petersen, W. Wylie, \textit{On the classification of gradient
Ricci solitons.} arXiv:0712.1298.

\bibitem {PW_rig}P. Petersen, W. Wylie, \textit{Rigidity of gradient Ricci
solitons.} Pacific J. Math.  \textbf{241} (2009), no. 2, 329--345.

\bibitem {PRiS}S. Pigola, M. Rimoldi, A.G. Setti, \textit{Remarks on non-compact gradient Ricci solitons.} To appear in Math. Z. arXiv:0905.2868v3.

\bibitem{Q}Z. Qian, \textit{Estimates for the weighted volumes and applications.} Quart. J. Math. Oxford Ser. (2) \textbf{48} (1997), no.190, pp. 235--242.

\bibitem{S}A.G. Setti, \textit{Eigenvalue estimates for the weighted Laplacian on a Riemannian manifold.} Rend. Sem. Mat. Univ. Padova \textbf{100} (1998) 27--55.

\bibitem {Y}S.T. Yau, \textit{Some function theoretic properties of complete
Riemannian manifolds and their applications to geometry.} Indiana Univ. Math.
J. \textbf{25} (1976), 659--670.
\end{thebibliography}
\end{document}